\documentclass{article}
\usepackage[affil-it]{authblk}
\usepackage[utf8]{inputenc}
\usepackage{amsmath, amsfonts, amsthm, amssymb, latexsym}
\usepackage{geometry}
\usepackage{enumerate}
\usepackage{multirow}
\usepackage{rotating}
\usepackage{graphicx,color}
\usepackage{float}
\usepackage{mathrsfs}

\usepackage{xargs}
\usepackage[pdftex,dvipsnames]{xcolor}

\usepackage{hyperref}
\hypersetup{colorlinks=true, pdfstartview=FitV,linkcolor=blue!70!black,citecolor=red!70!black, urlcolor=green!60!black}
\definecolor{labelkey}{rgb}{0.6,0,0}

\usepackage[colorinlistoftodos,prependcaption,textsize=small]{todonotes}
\newcommandx{\change}[2][1=]{\todo[#1]{#2}}
\newcommandx{\unsure}[2][1=]{\todo[linecolor=red,backgroundcolor=red!25,bordercolor=red,#1]{#2}}
\newcommandx{\rmk}[2][1=]{\todo[linecolor=blue,backgroundcolor=blue!25,bordercolor=blue,#1]{#2}}
\newcommandx{\info}[2][1=]{\todo[linecolor=OliveGreen,backgroundcolor=OliveGreen!25,bordercolor=OliveGreen,#1]{#2}}
\newcommandx{\improvement}[2][1=]{\todo[linecolor=Plum,backgroundcolor=Plum!25,bordercolor=Plum,#1]{#2}}
\newcommandx{\thiswillnotshow}[2][1=]{\todo[disable,#1]{#2}}

\providecommand{\abs}[1]{\left\vert#1\right\vert}

\newtheorem{thm}{Theorem}[section]

\newtheorem{prop}[thm]{Proposition}

\theoremstyle{definition}

\theoremstyle{remark}

\def\ud{\mathrm{d}}

\def\r{\mathbb{R}}
\def\no{\nonumber}
\def\ue{\mathrm{e}}

\title{Determining the collision kernel in the Boltzmann equation near the equilibrium}

\author[1]{Li Li \thanks{lili@ipam.ucla.edu}}
\affil[1]{Institute for Pure and Applied Mathematics, University of California,\protect\\ Los Angeles, CA 90095, USA}

\author[1]{Zhimeng Ouyang \thanks{zouyang@ipam.ucla.edu}}

\date{}

\begin{document}

\maketitle

\noindent \textbf{ABSTRACT.}\, We consider an inverse problem for the nonlinear Boltzmann equation near the equilibrium. Our goal is to determine the collision kernel in the Boltzmann equation from the knowledge of the Albedo operator. Our approach relies on a 
linearization technique as well as the injectivity of the Gauss-Weierstrass transform.

\smallskip
\section{Introduction} \label{Sec:Intro}

We consider the following evolutionary Boltzmann equation
\begin{equation}\label{Boltz}
\partial_t F+ v\cdot \nabla_x F= Q(F, F).
\end{equation}
Here $F(t, x, v)$ is the kinetic distribution function and the collision operator $Q$
is defined by
\begin{equation}\label{collQ}
Q(F_1, F_2):= \int_{\mathbb{R}^3}\int_{\mathbb{S}^2}
q(\theta, |v-u|)\big[F_1(u')F_2(v')-F_1(u)F_2(v)\big]\ud\omega \ud u
\end{equation}
where the vectors
\begin{equation}\label{uv}
u^{\prime}=u-[(u-v) \cdot \omega] \omega, \qquad v^{\prime}=v+[(u-v) \cdot \omega] \omega
\end{equation}
are velocities after a collision of particles with original velocities $u, v$ and $\theta\in [0, \frac{\pi}{2}]$ satisfies
\begin{equation}\label{cos}
\cos\theta= \frac{|(v-u)\cdot \omega|}{|v-u|}.
\end{equation}
The collision operator $Q$ describes the particle interaction and $q$ is called the collision kernel (or collision cross section). In this paper, we focus on $q$ which has the form 
\begin{equation}\label{qform}
q(\theta, |v-u|)= |v-u|^\gamma q_0(\theta),
\end{equation}
where the constant $\gamma$ satisfies $0\leq \gamma \leq 1$ (hard potential) and the smooth function $q_0$ satisfies $0\leq q_0(\theta)\leq C\cos\theta$ (angular cutoff).
This is an assumption introduced by Grad (see e.g. \cite{grad1958kinetic}) to tame the singularity of the collision kernel at $\theta=0$ and is one of the most well-accepted models.

To formulate our inverse problem, we consider the initial (in-flow) boundary value problem
\begin{equation}\label{ibvp}
\begin{cases}\;\partial_t F+ v\cdot \nabla_x F= Q(F, F) \quad& \mathbb{R}_+\times\Omega\times\mathbb{R}^{3}\\ 
\;F=G & \mathbb{R}_+\times\Gamma_{-}\\
\;F(0, x, v)= \mu & \Omega\times\mathbb{R}^{3}\end{cases}
\end{equation}
and we formally define the Albedo operator
\begin{equation}\label{Alop}
\mathcal{A}: G\to F|_{\mathbb{R}_+\times\Gamma_{+}}.
\end{equation}
Here the Gaussian function (normalized Maxwellian) $\mu$ is defined by 
\begin{equation}\label{Gf}
\mu(v):= \ue^{-|v|^2},
\end{equation}
$\mathbb{R}_+:=\{t: t>0\}$, $\Omega$ is a bounded, strictly convex domain with smooth boundary and 
$$\Gamma_\pm:= \big\{(x, v)\in \partial\Omega\times\mathbb{R}^3: \pm n(x)\cdot v> 0 \big\}$$
where $n(x)$ is the unit outer normal to $\partial\Omega$ at $x\in \partial\Omega$.

We will see that (\ref{ibvp}) is well-posed for continuous $G$ which is a small perturbation around the equilibrium $\mu$
so (\ref{Alop}) is well-defined for such $G$.
The following theorem is our main result.
\begin{thm} \label{Thm:main-1}
Let $\mathcal{A}^{(1)}, \mathcal{A}^{(2)}$ be the Albedo operators corresponding to
the collision kernels $q^{(1)}, q^{(2)}$ satisfying (\ref{qform}) and an additional symmetric assumption. Suppose $\mathcal{A}^{(1)}= \mathcal{A}^{(2)}$. Then $q^{(1)}= q^{(2)}$.
\end{thm}

We will provide a precise statement of the main theorem in Section 4 after we introduce more definitions and notations in later sections.

\subsection{Connection with earlier literature}

So far there have been many contributions in the mathematical study of different aspects of the forward problem for the Boltzmann equation. See e.g. \cite{cercignani1988boltzmann,cercignani1994dilute,glassey1996cauchy,ukai1986boltzmann}.
In the regime of bounded domains with physical boundary conditions,
the Boltzmann equation with angular cutoff has been proved by Guo to be globally well-posed and stable for small data near the Maxwellian equilibrium state for all four basic types of boundary conditions (see \cite{Guo2010}).
Other related results can be found in \cite{guiraud1975Htheorem,shizuta1977boltzmann,DesvillettesVillani2005boltzmann}.
On the other hand, the global well-posedness in bounded domains for the model without angular cutoff and for general solutions that are far from equilibrium (e.g. near vacuum) are completely open.

Inverse problems for linear transport equations have been extensively studied as well. We refer readers to \cite{bal2009inverse} for a survey on this topic. The inverse problem is to determine optical parameters from the knowledge of the Albedo operator associated with the linear Boltzmann equation (radiative transfer equation). Based on the singular decomposition of the Albedo operator, uniqueness results for inverse problems for the linear evolutionary Boltzmann equation have been obtained in \cite{choulli1996inverse}. The stationary case has been studied in \cite{choulli1999inverse} and the related stability estimates have been obtained (see e.g. \cite{lai2019inverse, wang1999stability}).

Fewer uniqueness results for inverse problems for the nonlinear Boltzmann equation have been obtained yet. An inverse problem for the nonlinear relativistic Boltzmann equation was studied in \cite{balehowsky2020inverse}, where the authors showed that the Lorentzian spacetime can be determined from the associated source-to-solution map up to an isometry for a fixed collision kernel.
For determining the collision kernel, a more related work is \cite{lai2021reconstruction}. In \cite{lai2021reconstruction}, the authors studied an inverse problem for the nonlinear stationary Boltzmann equation near the vacuum. They proved that the collision kernel can be determined from the associated Albedo operator under appropriate assumptions.
Compared with our $q$ in (\ref{qform}), the collision kernel studied in \cite{lai2021reconstruction} has a more general form. The main restriction in \cite{lai2021reconstruction} is the strong $L^1$ bound condition (see (1.5) in \cite{lai2021reconstruction}), which excludes the most classical hard sphere case
\begin{equation}\label{hsphere}
q(\theta, |v-u|)= c|v-u|\cos\theta
\end{equation}
(a special case of (\ref{qform})) arsing in the kinetic theory. We mention that the arguments in both \cite{balehowsky2020inverse} and \cite{lai2021reconstruction}
depend on the higher order multiple-fold linearization method introduced in
\cite{kurylev2018inverse}. This method has wide applications in solving inverse problems for nonlinear equations. See e.g.
\cite{feizmohammadi2020inverse, krupchyk2020remark, lassas2020partial, lassas2018inverse, li2021inverse, uhlmann2021inverse}.

Instead of the multiple-fold linearization method, a first order linearization method will be applied in this paper. This enables us to relate our problem to the one studied in \cite{choulli1996inverse}. The key point is that the information of the collision kernel in the nonlinear equation is encoded in the parameters in the associated linear equation. Thus we will be able to  determine the collision kernel once we apply the uniqueness result in \cite{choulli1996inverse} to determine the parameters.

\subsection{Organization}

The rest of this paper is organized in the following way. For later use, we will review some basic theories of the linear transport equation and the nonlinear Boltzmann equation near the equilibrium in Section~\ref{Sec:Prelim}. Based on the arguments in \cite{Guo2010},
we will show the well-posedness of the forward problem and relate our nonlinear problem to the linear one studied in \cite{choulli1996inverse} in Section~\ref{Sec:WPL}. In section~\ref{Sec:IP}, we will first determine the parameters in the linear equation based on the result in \cite{choulli1996inverse}. Then we will explicitly present our main theorem and further determine the collision kernel based on the injectivity of the Gauss-Weierstrass transform.
\medskip

\noindent \textbf{Acknowledgements.} L.L. and Z.O. are partly supported by the Simons Foundation. L.L. would like to thank Professor Gunther Uhlmann for helpful discussions.

\smallskip
\section{Preliminaries} \label{Sec:Prelim}
\subsection{Linear transport equation} \label{SubSec:Transport}
Let $\nu(v)$ be a positive function such that $\nu h_0\in L^1(\Omega\times \mathbb{R}^3)$ for any continuous $h_0$ compactly supported in $\Omega\times \mathbb{R}^3$. Let $K$ be a bounded linear integral operator on $L^1(\Omega\times \mathbb{R}^3)$ corresponding to a positive symmetric kernel $k(v, v')$, i.e. 
$$(Kh_0)(x, v')= \int_{\r^3} k(v, v')h_0(x, v)\,\ud v.$$

It is known that the semigroup $U_j(t): h_0\to h$ ($j= 0, 1, 2$)
associated with
\begin{equation}\label{semigpU}
\begin{cases}\;\partial_t h+ v\cdot \nabla_x h+ L_j h= 0 \quad& \mathbb{R}_+\times\Omega\times\mathbb{R}^{3}\\ 
\;h= 0 & \mathbb{R}_+\times\Gamma_{-}\\
\;h(0, x, v)= h_0 & \Omega\times\mathbb{R}^{3}\end{cases}
\end{equation}
is strongly continuous on 
$L^1(\Omega\times \mathbb{R}^3)$ where we define 
$$L_0:= 0,\qquad L_1h_0:= \nu h_0,\qquad L_2:= L_1- K.$$ 
(In fact this holds for more general position-dependent $\nu, k$. See e.g. \cite[Theorem~1]{vidav1968existence}.)  Clearly,
$$U_1(t)h_0= 1_{t\leq \tau_-(x,v)}\ue^{-\nu(v)t}h_0(x-tv, v)
= \ue^{-\nu(v)t}U_0(t)h_0,$$
where $\tau_-$ is the exit time function defined by
\begin{equation}\label{exit}
\tau_-(x,v):=\sup\big\{t\geq 0: x-tv\in \Omega\big\},
\end{equation}
and by Duhamel's principle we have
$$U_2(t)= U_1(t)+ \int^t_0 U_2(t-s)KU_1(s)\,\ud s.$$

Let $G_-(t): g\to h$ denote the solution operator associated with
\begin{equation}\label{ibvpL1g}
\begin{cases}\;\partial_t h+ v\cdot \nabla_x h+ L_1 h= 0 \quad& \mathbb{R}_+\times\Omega\times\mathbb{R}^{3}\\ 
\;h= g & \mathbb{R}_+\times\Gamma_{-}\\
\;h(0, x, v)= 0 & \Omega\times\mathbb{R}^{3}.\end{cases}
\end{equation}
We have
$$G_-(t)g= \ue^{-\nu(v)\tau_-(x,v)}g(t- \tau_-(x,v), x- \tau_-(x,v)v, v).$$
(We define $g(t):= 0$ for $t\leq 0$.) It is known that
\begin{equation}\label{GgL1}
\sup_t\big\|G_-(t)g\big\|_{L^1(\Omega\times \mathbb{R}^3)}\leq 
\|g\|_{L^1(\mathbb{R}_+\times\Gamma_{-},\,|n(x)\cdot v|\ud t\ud\sigma(x)\ud v)}
\end{equation}
where $d\sigma$ is the standard surface measure on $\partial\Omega$.
(See e.g. (5.5) in \cite{choulli1996inverse}.)
By Duhamel's principle, we know that
the solution of
\begin{equation}\label{ibvpL2g}
\begin{cases}\;\partial_t h+ v\cdot \nabla_x h+ L_2 h= 0 \quad& \mathbb{R}_+\times\Omega\times\mathbb{R}^{3}\\ 
\;h= g & \mathbb{R}_+\times\Gamma_{-}\\
\;h(0, x, v)= 0 & \Omega\times\mathbb{R}^{3}\end{cases}
\end{equation}
is given by the formula
\begin{align}
    h&= G_-(t)g+ \int^t_0 U_2(s)KG_-(t-s)g\,\ud s \label{soldecomp}\\
    &= G_-(t)g+ \int^t_0 U_1(s)KG_-(t-s)g\,\ud s + \int^t_0\int^{s_2}_0 U_2(t-s_2)KU_1(s_1)KG_-(s_2-s_1)g\,\ud s_1\ud s_2.\no
\end{align}

\subsection{Boltzmann equation near the equilibrium} \label{SubSec:Boltzmann}
All materials in this subsection can be found in Chapter 3 in \cite{glassey1996cauchy}.

By making the substitutions
\begin{equation}\label{FGfg}
F= \mu+ \mu^{\frac{1}{2}}f,\qquad G= \mu+ \mu^{\frac{1}{2}}g,
\end{equation}
we can write (\ref{ibvp}) as 
\begin{equation}\label{ibvpfg}
\begin{cases}\;\partial_t f+ v\cdot \nabla_x f+ Lf= \Gamma(f, f) \quad& \mathbb{R}_+\times\Omega\times\mathbb{R}^{3}\\ 
\;f=g & \mathbb{R}_+\times\Gamma_{-}\\
\;f(0, x, v)= 0 & \Omega\times\mathbb{R}^{3},\end{cases}
\end{equation}
where
\begin{equation}\label{muQ}
\Gamma(f, f)= \mu^{-\frac{1}{2}}Q(\mu^{\frac{1}{2}}f, \mu^{\frac{1}{2}}f),
\end{equation}
and the linearized Boltzmann operator $L$ has the form
$L= \nu - K$. 
Here the function $\nu(v)$ is the collision frequency and $K$ has the form $K= K_2- K_1$ where $K_1, K_2$ are
the linear integral operators corresponding to the kernels $k_1, k_2$. 
It is known that 
\begin{equation}\label{nu}
\nu(v)= \int_{\mathbb{R}^3}\int_{\mathbb{S}^2} q(\theta, |v-u|)\mu(u)\,\ud\omega \ud u,
\end{equation}
\begin{equation}\label{k1}
k_1(u, v)= \mu^{\frac{1}{2}}(u)\mu^{\frac{1}{2}}(v)\int_{\mathbb{S}^2} q(\theta, |v-u|)\,\ud\omega,
\end{equation}
\begin{equation}\label{k2}
k_2(u, v)= \frac{2}{|u- v|^2}\ue^{-\frac{|u- v|^2}{4}}\int_{y\in \Pi}\ue^{-|y+\zeta|^2}\tilde{q}(|u-v|, |y|) \,\ud\Pi
\end{equation}
where 
\begin{equation}\label{zetaPi}
\zeta= \frac{1}{2}(v+ u),\qquad \Pi= \big\{y: y\cdot(u- v)= 0\big\}
\end{equation}
and the function $\tilde{q}(\cdot, \cdot)$ is defined by
\begin{equation}\label{tildeq}
\tilde{q}(\rho\cos\theta, \rho\sin\theta):=
\frac{\tilde{B}(\theta, \rho)}{\sin\theta},\quad
\tilde{B}(\theta, \rho)= \frac{1}{2}\big[B(\theta, \rho)+
B(\tfrac{\pi}{2}-\theta, \rho)\big],\quad B(\theta, \rho)= 
q(\theta, \rho)\sin\theta.
\end{equation}

\smallskip
\section{Well-posedness and linearization} \label{Sec:WPL}
\subsection{Well-posedness} \label{SubSec:WP}

In order to establish well-posedness for the forward problem via the $L^\infty$ framework,
we introduce a weight function which has the form
\begin{equation}\label{weight}
w(v)= (1+ c|v|^2)^m,\qquad c>0,\,\,m\in \mathbb{R}
\end{equation}
satisfying $w^{-2}(1+ |v|)^3\in L^1(\mathbb{R}^3)$ (see Subsection 1.3 in \cite{Guo2010}).
By making the substitutions
$$\tilde{f}= wf,\qquad \tilde{g}= wg$$ 
in (\ref{ibvpfg}), we can further write (\ref{ibvp}) as 
\begin{equation}\label{ibvptildefg}
\begin{cases}\;\partial_t \tilde{f}+ v\cdot \nabla_x \tilde{f}+ \tilde{L}\tilde{f}= \tilde{\Gamma}(\tilde{f}, \tilde{f}) \quad& \mathbb{R}_+\times\Omega\times\mathbb{R}^{3}\\ 
\;\tilde{f}= \tilde{g} & \mathbb{R}_+\times\Gamma_{-}\\
\;\tilde{f}(0, x, v)= 0 & \Omega\times\mathbb{R}^{3}\end{cases}
\end{equation}
where
\begin{equation}\label{tildeL}
\tilde{L}:= \nu- \tilde{K},\qquad \tilde{K}:= wK(\tfrac{\cdot}{w}),
\end{equation}
\begin{equation}\label{tildeGamma}
\tilde{\Gamma}(\cdot, \cdot):= w\Gamma(\tfrac{\cdot}{w}, \tfrac{\cdot}{w}).
\end{equation}
Based on \cite[Theorem~1]{Guo2010}, we have the following well-posedness result for (\ref{ibvptildefg}).

\begin{prop}\label{Prop:WP}
For $\tilde{g}$ with sufficiently small $L^\infty$-norm, (\ref{ibvptildefg}) has a unique solution $\tilde{f}$ and
\begin{equation}\label{Linftyfg}
\|\tilde{f}\|_{L^\infty}\leq C\|\tilde{g}\|_{L^\infty}.
\end{equation}
Moreover, if $\tilde{g}$ is continuous on $[0, \infty)\times\Gamma_{-}$,
then $\tilde{f}$ is continuous in $[0, \infty)\times\big\{(\bar{\Omega}\times\mathbb{R}^3)\setminus\Gamma_0\big\}$,
where 
\begin{equation}\label{Gamma0}
\Gamma_0:= \big\{(x, v)\in \partial\Omega\times\mathbb{R}^3: n(x)\cdot v= 0 \big\}.
\end{equation}
\end{prop}
We remark that 
the $L^\infty$ estimate above is an adaption from \cite[Theorem~1]{Guo2010} without time decay.
Also, this continuity result for the in-flow boundary requires a strictly convex domain.
Alternatively, if we make sense of the boundary restriction map using Ukai's trace theorem (see \cite[Theorem~5.5.1]{ukai1986boltzmann}), then we may also work with non-convex domains and $L^p$ boundary data.

Hence we know that the associated Albedo operator
\begin{equation}\label{Aloptilde}
A: \tilde{g}\to \tilde{f}|_{\mathbb{R}_+\times\Gamma_{+}}
\end{equation}
is at least 
well-defined for small continuous $\tilde{g}$.

Clearly the knowledge of $A$ is equivalent to the knowledge of $\mathcal{A}$
defined by (\ref{Alop}).

\subsection{Linearization}
Let $\tilde{f_\epsilon}$ be the solution of 
\begin{equation}\label{smallibvp}
\begin{cases}\;\partial_t \tilde{f}+ v\cdot \nabla_x \tilde{f}+ \tilde{L}\tilde{f}= \tilde{\Gamma}(\tilde{f}, \tilde{f}) \quad& \mathbb{R}_+\times\Omega\times\mathbb{R}^{3}\\ 
\;\tilde{f}= \epsilon\tilde{g} & \mathbb{R}_+\times\Gamma_{-}\\
\;\tilde{f}(0, x, v)= 0 & \Omega\times\mathbb{R}^{3}\end{cases}
\end{equation}
for small $\epsilon$ and continuous $\tilde{g}$. Consider the linear problem
\begin{equation}\label{linearibvp}
\begin{cases}\;\partial_t h+ v\cdot \nabla_x h+ \tilde{L}h= 0 \quad& \mathbb{R}_+\times\Omega\times\mathbb{R}^{3}\\ 
\;h= \tilde{g} & \mathbb{R}_+\times\Gamma_{-}\\
\;h(0, x, v)= 0 & \Omega\times\mathbb{R}^{3}.\end{cases}
\end{equation}

\begin{prop}\label{Prop:Linearization}
$\frac{\tilde{f_\epsilon}}{\epsilon}\to h$ 
in $L^\infty$-norm as $\epsilon\to 0$.
\end{prop}
\begin{proof}
Let $h_\epsilon:= \frac{\tilde{f_\epsilon}}{\epsilon}- h$.
Note that we have $h_\epsilon(0)= 0,\, h_\epsilon|_{\mathbb{R}_+\times\Gamma_{-}}= 0$
and
$$\partial_t h_\epsilon+ v\cdot \nabla_x h_\epsilon+ \tilde{L}h_\epsilon= \frac{1}{\epsilon}\tilde{\Gamma}(\tilde{f_\epsilon}, \tilde{f_\epsilon}).$$

We will show that $h_\epsilon\to 0$ in $L^\infty$-norm as $\epsilon\to 0$. 

Let $\tilde{U}(t): h_0\to h$ ($j= 0, 1, 2$) denote the semigroup associated with
\begin{equation}\label{tildesemigp}
\begin{cases}\;\partial_t h+ v\cdot \nabla_x h+ \tilde{L} h= 0 \quad& \mathbb{R}_+\times\Omega\times\mathbb{R}^{3}\\ 
\;h= 0 & \mathbb{R}_+\times\Gamma_{-}\\
\;h(0, x, v)= h_0 & \Omega\times\mathbb{R}^{3}.\end{cases}
\end{equation}
Let $\tilde{k}$ be the kernel corresponding to the linear integral operator 
$\tilde{K}$ in (\ref{tildeL}).
Based on the estimates (44), (45) in \cite[Lemma~3]{Guo2010} (these stronger estimates are mainly used to prove the weighted $L^\infty$ bounds of solutions), 
we have
\begin{equation}\label{tildekbdd}
\int_{\r^3}\tilde{k}(v, v')\,\ud v\leq C,
\end{equation}
which implies $\tilde{K}$ is bounded on $L^1(\Omega\times \mathbb{R}^3)$. 
Based on results in Subsection~\ref{SubSec:Transport} and Duhamel's principle, we know that
$\tilde{U}(t)$ is strongly continuous on $L^1(\Omega\times \mathbb{R}^3)$ 
and we have
\begin{align}
    h_\epsilon(t)&=\int_0^t\tilde{U}(t-s)\frac{1}{\epsilon}\tilde{\Gamma}(\tilde{f_\epsilon}, \tilde{f_\epsilon})(s)\,\ud s\label{hepsilon}\\
    &=\int_0^tU_1(t-s)\frac{1}{\epsilon}\tilde{\Gamma}(\tilde{f_\epsilon}, \tilde{f_\epsilon})(s)\,\ud s+
    \int_0^t\int_s^tU_1(t-s')\tilde{K}\tilde{U}(s'-s)\frac{1}{\epsilon}\tilde{\Gamma}(\tilde{f_\epsilon}, \tilde{f_\epsilon})(s)\,\ud s'\ud s.\no
\end{align}
Based on the estimate (233) in \cite{Guo2010}, we have 
\begin{equation}\label{hepsilonint1}
\abs{\int_0^tU_1(t-s)\frac{1}{\epsilon}\tilde{\Gamma}(\tilde{f_\epsilon}, \tilde{f_\epsilon})(s) \,\ud s}\leq C\epsilon^{-1}\|\tilde{f_\epsilon}\|^2_{L^\infty}.
\end{equation}
Based on the estimate (237) in \cite{Guo2010}, we have 
\begin{equation}\label{hepsilonint2}
\abs{\int_0^t\int_s^tU_1(t-s')\tilde{K}\tilde{U}(s'-s)\frac{1}{\epsilon}\tilde{\Gamma}(\tilde{f_\epsilon}, \tilde{f_\epsilon})(s)\,\ud s'\ud s}\leq C'\epsilon^{-1}\|\tilde{f_\epsilon}\|^2_{L^\infty}.
\end{equation}
Hence by (\ref{hepsilonint1}), (\ref{hepsilonint2}) and (\ref{Linftyfg}) we have
$$\|h_\epsilon\|_{L^\infty}\leq C''\epsilon^{-1}\|\tilde{f_\epsilon}\|^2_{L^\infty}
\leq C'''\epsilon\|\tilde{g}\|^2_{L^\infty}.$$
\end{proof}

Now we consider 
the Albedo operator
\begin{equation}\label{linAlop}
A^{lin}: \tilde{g}\to h|_{\mathbb{R}_+\times\Gamma_{+}}
\end{equation}
associated with the linear problem (\ref{linearibvp}). 

We take the restriction to $\mathbb{R}_+\times\Gamma_{-}$ in Proposition~\ref{Prop:Linearization} to obtain that
\begin{equation}\label{AtoAlin}
\frac{1}{\epsilon}A(\epsilon\tilde{g})\to A^{lin}\tilde{g}
\end{equation}
in $L^\infty$-norm as $\epsilon\to 0$ for continuous $\tilde{g}$ compactly supported in $\mathbb{R}_+\times\Gamma_{-}$, which implies that
$A^{lin}$ is determined by $A$.

\smallskip
\section{Inverse problem} \label{Sec:IP}
\subsection{Determine the collision frequency and K}
Based on the formula (\ref{soldecomp}), we have the following singular decomposition result for the Albedo operator $A^{lin}$ associated with the linear problem (\ref{linearibvp}). See \cite[Theorem~5.1]{choulli1996inverse}.

\begin{prop}
The Schwartz kernel of $A^{lin}$ has the form $\alpha(t-t',x,v,x',v')$,
i.e. formally 
$$(A^{lin}\tilde{g})(t,x,v)= \iint_{\Gamma_-}\int_{\r_+}\alpha(t-t',x,v,x',v')\tilde{g}(t',x',v')\,\ud t'\ud\sigma(x')\ud v'.$$
We have the decomposition $\alpha= \alpha_0+ \alpha_1+ \alpha_2$ where
$$\alpha_j= \alpha_j(\tau,x,v,x',v'),\qquad (x,v)\in \Gamma_+,\,\,
(x', v')\in \Gamma_-,$$
$$\alpha_0= \ue^{-\nu(v)\tau_-(x,v)}\delta_{\{x-\tau_-(x,v)v\}}(x')\delta(v-v')
\delta_1(\tau-\tau_-(x,v)),$$
$$\alpha_1= \int^{\tau_-(x,v)}_0 \ue^{-\nu(v)s}\ue^{-\nu(v')\tau_-(x-sv,v')}\delta_1(\tau-s-\tau_-(x-sv,v'))
\tilde{k}(v, v')\delta_{\{x-sv-\tau_-(x-sv,v')v'\}}(x')\,\ud s,$$
and $\alpha_2$ satisfies
$$|n(x')\cdot v'|^{-1}\alpha_2\in L^\infty(\Gamma_-; L^1_{\mathrm{loc}}(\mathbb{R}; L^1(\Gamma_+, |n(x)\cdot v|\,\ud\sigma(x)\ud v))).$$
Here we use $\delta, \delta_1$ to denote the standard Dirac distribution on $\mathbb{R}^3, \mathbb{R}$. For $y\in \partial\Omega$, $\delta_y$ is the distribution defined by $\langle\delta_y, \varphi\rangle:= \varphi(y)$
for $\varphi$ defined on $\partial\Omega$.
\end{prop}

We remark that $\alpha_0, \alpha_1$ are singular distributions while $\alpha_2$ is a function. $\alpha_0$ is a Dirac type distribution, which is supported at a point for fixed $(x, v)$. $\alpha_1$ is a Dirac type distribution as well but it is less singular than $\alpha_0$. For fixed $(x, v, v')$, the support of $\alpha_1$ is contained in the set $\{(x', \tau)\}$ where $x'$ belongs to the intersection curve of $\partial\Omega$ with the plane passing through $x$ and parallel to $v, v'$, and $\tau$ is the travel time from $x$ to $x'$. 

Based on the decomposition theorem above, we can use exactly the same method presented in \cite{choulli1996inverse} to obtain the following uniqueness result for the linear problem (\ref{linearibvp}). Recall that
$\nu, K$ are defined in Subsection~\ref{SubSec:Boltzmann}. The knowledge of $K$ is equivalent to the knowledge of $\tilde{K}$ and $\nu, \tilde{K}$ are parameters appearing in the linear problem (\ref{linearibvp}) (see (\ref{tildeL}) in Subsection~\ref{SubSec:WP}).

\begin{prop}\label{Prop:IP-1}
Both $\nu$ and $K$ are uniquely determined by $A^{lin}$.
\end{prop}

We remark that in \cite{choulli1996inverse}, the authors considered position-dependent parameters so it is only possible to determine the X-ray transform of $\nu(x, v)$ from $A^{lin}$. Here we are only interested in position-independent $\nu$ and $K$ so $\nu(v)$ can be uniquely determined. 

Now we sketch the proof and we refer readers to Section 5 in \cite{choulli1996inverse} for details.

\begin{proof}
For fixed $(x, v)$, we can appropriately choose $\phi_\epsilon(\tau, x', v')$ based on the support of $\alpha_0$ (see the expression above (5.17) in \cite{choulli1996inverse}) such that
$$\lim_{\epsilon\to 0}\iint_{\Gamma_+}\int_{\r_+}\alpha_0(\tau, x, v, x', v')\phi_\epsilon(\tau, x', v')\,\ud\tau \ud\sigma(x') \ud v'= \ue^{-\tau_-(x,v)\nu(v)},$$
$$\lim_{\epsilon\to 0}\iint_{\Gamma_+}\int_{\r_+}\alpha_j(\tau, x, v, x', v')\phi_\epsilon(\tau, x', v')\,\ud\tau \ud\sigma(x') \ud v'= 0,\qquad
j= 1,2,$$
so the action of $\alpha$ on $\phi_\epsilon$ gives the reconstruction formula of $\nu$ from $\alpha$.

For fixed $(x, t, v, v')$ with $v\neq  v'$, we have $\alpha_0 = 0$ and we can appropriately choose $\psi_{\epsilon_1, \epsilon_2}$ based on the support of $\alpha_1$ (see (3.9), (5.21) and (5.22) in \cite{choulli1996inverse})
such that
$$\lim_{\epsilon_1,\epsilon_2\to 0}\int_{\r_+}\int_{\partial\Omega}
\alpha_1(t-t',x,v,x',v')\psi_{\epsilon_1, \epsilon_2}(x', t')\,\ud\sigma(x') \ud t'
= \ue^{-t\nu(v)}\ue^{-\tau_-(x-tv,v')\nu(v')}\tilde{k}(v, v'),$$
$$\lim_{\epsilon_1,\epsilon_2\to 0}\int_{\r_+}\int_{\partial\Omega}
\alpha_2(t-t',x,v,x',v')\psi_{\epsilon_1, \epsilon_2}(x', t')\,\ud\sigma(x') \ud t'
= 0,$$
so the action of $\alpha$ on $\psi_{\epsilon_1, \epsilon_2}$ gives the reconstruction formula of $\tilde{k}$ (equivalent to $K$) from $\alpha$ once $\nu$ is reconstructed.
\end{proof}

\subsection{Determine the collision kernel}
Now we are ready to explicitly present our main theorem. Our result depends on the uniqueness result (Proposition~\ref{Prop:IP-1}) for the linear problem (\ref{linearibvp}) as well as the injectivity of the Gauss-Weierstrass transform (convolution with the Maxwellian $\mu$). The following theorem is the precise version of Theorem~\ref{Thm:main-1}. Recall that the Albedo operator $A$ is defined by (\ref{Aloptilde}) and $\tilde{q}, B$ are defined by (\ref{tildeq}).

\begin{thm}
Let $A^{(1)}, A^{(2)}$ be the Albedo operators corresponding to
the collision kernels $q^{(1)}, q^{(2)}$ satisfying (\ref{qform}).
Suppose 
\begin{equation}\label{A1A2}
A^{(1)}\tilde{g}= A^{(2)}\tilde{g}
\end{equation}
for any continuous $\tilde{g}$ compactly supported in $\mathbb{R}_+\times\Gamma_{-}$. Then 
$$\tilde{q}^{(1)}= \tilde{q}^{(2)}.$$
If we further assume
\begin{equation}\label{symB}
B(\theta, \rho)=
B(\tfrac{\pi}{2}-\theta, \rho),
\end{equation}
then $\tilde{q}$ is just the Cartesian representation of $q$ so we can conclude that 
$$q^{(1)}= q^{(2)}$$
in this case. (e.g. the hard sphere case (\ref{hsphere}) satisfies (\ref{symB}).)
\end{thm}

\begin{proof}
By (\ref{AtoAlin}) and the assumption (\ref{A1A2}) we have 
$$(A^{lin})^{(1)}= (A^{lin})^{(2)},$$
and then by Proposition~\ref{Prop:IP-1} we have 
$$\nu^{(1)}= \nu^{(2)},\qquad K^{(1)}= K^{(2)}.$$

We define the function
$$I(z):= \int_{\mathbb{S}^2} q\big(\cos^{-1}(\tfrac{|z\cdot \omega|}{|z|}), |z|\big)\,\ud\omega.$$
Now we can write (\ref{nu}) as 
$$\nu^{(j)}(v)= \big(I^{(j)}* \mu\big)(v)$$
so $\nu^{(1)}= \nu^{(2)}$ implies 
that $I^{(1)}= I^{(2)}$ based on the injectivity of the Gauss-Weierstrass transform. 

Since we can write (\ref{k1}) as 
$$k^{(j)}_1(u,v)= \mu^{\frac{1}{2}}(u)\mu^{\frac{1}{2}}(v)I^{(j)}(u-v),$$
we have $K^{(1)}_1= K^{(2)}_1$. Recall that $K^{(j)}= K^{(j)}_2- K^{(j)}_1$,
so $K^{(1)}= K^{(2)}$ implies that $K^{(1)}_2= K^{(2)}_2$.

Let $\eta:= u-v$. We can write (\ref{k2}) as
\begin{equation}\label{zetaeta}
k_2(\zeta, \eta)= \frac{2}{|\eta|^2}\ue^{-\frac{|\eta|^2}{4}}\int_{y\in \Pi}\ue^{-|y+\zeta|^2}\tilde{q}(|\eta|, |y|) \,\ud\Pi.
\end{equation}
Here we view $k_2$ as a function of the two new independent variables $\zeta, \eta$.
For each $\zeta\in \Pi$, the integral in (\ref{zetaeta}) is the value of the convolution of $\mu$ and $\tilde{q}(|\eta|, |\cdot|)$ over the plane $\Pi$ at $-\zeta$. Hence 
$K^{(1)}_2= K^{(2)}_2$ implies that $\tilde{q}^{(1)}= \tilde{q}^{(2)}$
based on the injectivity of the Gauss-Weierstrass transform.
\end{proof}

We remark that the proof above works for the general collision kernel $q(\theta, |v-u|)$. We restrict ourselves to $q$ which has the form (\ref{qform}) mainly because this assumption is required for the well-posedness of the forward problem (Proposition~\ref{Prop:WP}).

We mention that the 1-dimensional Gauss-Weierstrass transform is closely related with the Laplace transform. We also have an inversion formula involving Hermite polynomials for the general multi-dimensional Gauss-Weierstrass transform. We refer readers to Chapter 5 in \cite{brychkov1989integral} for details.

\medskip
\bibliographystyle{plain}
{\small\bibliography{Reference7}}

\begin{thebibliography}{10}

\bibitem{bal2009inverse}
Guillaume Bal.
\newblock Inverse transport theory and applications.
\newblock {\em Inverse Problems}, 25(5):053001, 2009.

\bibitem{balehowsky2020inverse}
Tracey Balehowsky, Antti Kujanp{\"a}{\"a}, Matti Lassas, and Tony Liimatainen.
\newblock An inverse problem for the relativistic {Boltzmann} equation.
\newblock {\em arXiv preprint arXiv:2011.09312}, 2020.

\bibitem{brychkov1989integral}
Yu~A Brychkov and A~P Prudnikov.
\newblock {\em Integral Transforms of Generalized Functions}.
\newblock Gordon \& Breach Sci. Publ, 1989.

\bibitem{cercignani1988boltzmann}
C.~Cercignani.
\newblock {\em The Boltzmann Equation and Its Applications}.
\newblock Springer, Berlin, 1988.

\bibitem{cercignani1994dilute}
C.~Cercignani, R.~Illner, and M.~Pulvirenti.
\newblock {\em The Mathematical Theory of Dilute Gases}.
\newblock Springer, Berlin, 1994.

\bibitem{choulli1996inverse}
Mourad Choulli and Plamen Stefanov.
\newblock Inverse scattering and inverse boundary value problems for the linear
  {Boltzmann} equation.
\newblock {\em Communications in Partial Differential Equations},
  21(5-6):763--785, 1996.

\bibitem{choulli1999inverse}
Mourad Choulli and Plamen Stefanov.
\newblock An inverse boundary value problem for the stationary transport
  equation.
\newblock {\em Osaka journal of mathematics}, 36(1):87--104, 1999.

\bibitem{DesvillettesVillani2005boltzmann}
L.~Desvillettes and C.~Villani.
\newblock On the trend to global equilibrium for spatially inhomogeneous
  kinetic systems: the {Boltzmann} equation.
\newblock {\em Invent. Math.}, 159(2):245--316, 2005.

\bibitem{feizmohammadi2020inverse}
Ali Feizmohammadi and Lauri Oksanen.
\newblock An inverse problem for a semi-linear elliptic equation in
  {Riemannian} geometries.
\newblock {\em Journal of Differential Equations}, 269(6):4683--4719, 2020.

\bibitem{glassey1996cauchy}
Robert~T Glassey.
\newblock {\em The Cauchy problem in kinetic theory}.
\newblock SIAM, 1996.

\bibitem{grad1958kinetic}
H.~Grad.
\newblock Principles of the kinetic theory of gases.
\newblock {\em Handbuch der Physik}, vol. XII:205--294, 1958.

\bibitem{guiraud1975Htheorem}
J.P. Guiraud.
\newblock An {H}-theorem for a gas of rigid spheres in a bounded domain.
\newblock {\em Theories cinetique classique et relativistes, CNRS, Paris},
  pages 29--58, 1975.

\bibitem{Guo2010}
Yan Guo.
\newblock Decay and continuity of the {Boltzmann} equation in bounded domains.
\newblock {\em Arch. Ration. Mech. Anal.}, 197:713--809, 2010.

\bibitem{krupchyk2020remark}
Katya Krupchyk and Gunther Uhlmann.
\newblock A remark on partial data inverse problems for semilinear elliptic
  equations.
\newblock {\em Proceedings of the American Mathematical Society},
  148(2):681--685, 2020.

\bibitem{kurylev2018inverse}
Yaroslav Kurylev, Matti Lassas, and Gunther Uhlmann.
\newblock Inverse problems for {Lorentzian} manifolds and non-linear hyperbolic
  equations.
\newblock {\em Inventiones mathematicae}, 212(3):781--857, 2018.

\bibitem{lai2019inverse}
Ru-Yu Lai, Qin Li, and Gunther Uhlmann.
\newblock Inverse problems for the stationary transport equation in the
  diffusion scaling.
\newblock {\em SIAM Journal on Applied Mathematics}, 79(6):2340--2358, 2019.

\bibitem{lai2021reconstruction}
Ru-Yu Lai, Gunther Uhlmann, and Yang Yang.
\newblock Reconstruction of the collision kernel in the nonlinear {Boltzmann}
  equation.
\newblock {\em SIAM Journal on Mathematical Analysis}, 53(1):1049--1069, 2021.

\bibitem{lassas2020partial}
Matti Lassas, Tony Liimatainen, Yi-Hsuan Lin, and Mikko Salo.
\newblock Partial data inverse problems and simultaneous recovery of boundary
  and coefficients for semilinear elliptic equations.
\newblock {\em Revista Matem{\'a}tica Iberoamericana}, 37(4):1553--1580, 2020.

\bibitem{lassas2018inverse}
Matti Lassas, Gunther Uhlmann, and Yiran Wang.
\newblock Inverse problems for semilinear wave equations on {Lorentzian}
  manifolds.
\newblock {\em Communications in Mathematical Physics}, 360(2):555--609, 2018.

\bibitem{li2021inverse}
Li~Li.
\newblock On inverse problems arising in fractional elasticity.
\newblock {\em arXiv:2109.03387, (to appear) Journal of Spectral Theory, 2022},
  2021.

\bibitem{shizuta1977boltzmann}
Y.~Shizuta and K.~Asano.
\newblock Global solutions of the {Boltzmann} equation in a bounded convex
  domain.
\newblock {\em Proc. Jpn. Acad.}, 53A:3--5, 1977.

\bibitem{uhlmann2021inverse}
Gunther Uhlmann and Jian Zhai.
\newblock On an inverse boundary value problem for a nonlinear elastic wave
  equation.
\newblock {\em Journal de Math{\'e}matiques Pures et Appliqu{\'e}es},
  153:114--136, 2021.

\bibitem{ukai1986boltzmann}
S.~Ukai.
\newblock Solutions of the {Boltzmann} equation.
\newblock {\em Patterns and waves}, Stud. Math. Appl., 18:37--96, 1986.

\bibitem{vidav1968existence}
Ivan Vidav.
\newblock Existence and uniqueness of nonnegative eigenfunctions of the
  {Boltzmann} operator.
\newblock {\em Journal of Mathematical Analysis and Applications},
  22(1):144--155, 1968.

\bibitem{wang1999stability}
Jenn-Nan Wang.
\newblock Stability estimates of an inverse problem for the stationary
  transport equation.
\newblock {\em Annales de l'IHP Physique th{\'e}orique}, 70(5):473--495, 1999.

\end{thebibliography}
\end{document}